\documentclass[11pt,a4paper,final]{amsart}

\usepackage{amsmath}
\usepackage{paralist}
\usepackage{graphics}
\usepackage{graphicx}
\usepackage{epstopdf}
\usepackage[colorlinks=true]{hyperref}
\usepackage{amssymb}
\usepackage{amsthm}
\usepackage{color}
\usepackage{multirow}
\usepackage{booktabs}
\usepackage{bm}
\usepackage{mathrsfs}

\usepackage[ruled,vlined]{algorithm2e}

\newtheorem{proposition}{Proposition}

\newtheorem{remark}{Remark}

\begin{document}

\title[Utility indifference Option Pricing Model]{
Utility indifference Option Pricing Model with a Non-Constant Risk-Aversion under Transaction Costs and Its Numerical Approximation}

\author{Pedro P\'olvora${}^{1}$}
\author{Daniel \v{S}ev\v{c}ovi\v{c}${}^{1}$}
\address{${}^{1}$ Department of Applied Mathematics and Statistics, Faculty of Mathematics Physics and Informatics, Comenius University, Mlynsk\'a dolina, 842 48, Bratislava, Slovakia. Corresponding author: {\tt sevcovic@fmph.uniba.sk} }

\begin{abstract}
Our goal is to analyze the system of Hamilton-Jacobi-Bellman equations arising in derivative securities pricing models. The European style of an option price is constructed as a difference of the certainty equivalents to the value functions solving the system of HJB equations. We introduce the transformation method for solving the penalized nonlinear partial differential equation. The transformed equation involves possibly non-constant the risk aversion function containing the negative ratio between the second and first derivatives of the utility function. Using comparison principles we derive useful bounds on the option price. We also propose a finite difference numerical discretization scheme with some computational examples. 

\medskip
\noindent
2010 MSC. Primary: 45K05 · 35K58 · 34G20 · 91G20

\medskip
\noindent Key words and phrases: option pricing; utility indifference pricing; transaction costs; Hamilton-Jacobi-Bellman equation; penalty methods; finite difference approximation

\end{abstract}

\maketitle

\section{Introduction}

In the last century the world witnessed a tremendous change and evolution in almost every industry and the financial one is no exception. One of the aspects that evolved greatly in finance was the financial derivatives, which saw their usage grow exponentially.

A financial derivative is a contract between two parties where they agree to future financial exchanges and whose value depends on one more underlying assets. There are multiple types of these contracts and they are used extensively in many industries, both for hedging and speculation. 
Depending on the type of derivative and on the position (buyer vs. seller) they can be used to either limit or increase the financial exposure to a particular financial asset. Examples of uses of financial derivatives include: financial institutions transforming a pool of equally risky mortgages into multiple contracts with different specific risk profiles, international enterprises reducing their foreign exchange risk, investors increasing their exposure to the increase in price of a stock by buying \mbox{financial derivatives. }

Financial options, a particular type of financial derivatives, are contracts where the buyer has the option but not the obligation to transact an asset at predefined conditions such as price or time. 
A key aspect of financial options is that their value, or price, is dependent on the underlying assets and finding it is fundamental for trading and managing the option and requires some type of mathematical modeling.  Since the future payoff of the option is uncertain at the time of the trade it is required to price it with probabilistic and statistical considerations, and different approaches have been developed, examples include: Binomial Trees, Monte Carlo simulations, Black-Scholes/PDEs . This was a key piece in the birth of what is now called financial mathematics. 

The usage of these type of contracts is very diverse, and they are important not only for the financial industry but for virtually every industry. They make it possible to manage risk in a way that gives great financial flexibility to enterprises, consequently promoting economic and business growth and development. For these reasons, it is not surprising that the number and type of financial options issued and traded has grown immensely over the last decades. This growth required more and more elaborate models to accommodate the new complexity of the contracts. Also, the greater size and impact of the usage of these contracts made evident the first models widely used were not good enough. Recent very unfortunate economic events such as the sub-prime crisis in the United States in 2008 were partially caused by the misuse of financial derivatives, and this demonstrated the importance of properly modeling and understand these contracts. The research proposed in this document aims to study new mathematical models that take into account some often neglected features of financial markets.

The most well known model for pricing financial options is called the Black-Scholes (BS) model and, although still largely used, it has multiple shortfalls like the fact that it does not account for feedback effects or transaction costs. The BS model does not consider that a trade can have an impact on the price, however, it has been empirically verified that a very large trader (such as an investment bank) can affect the assets’ prices upon performing a large trade. Also, virtually every market has transaction costs or a different price for buying versus selling an asset (bid-ask spreads), which is not considered by the BS model as it assumes continuous cost-free trading to perfectly hedge the portfolio. Due to these shortfalls, in many situations, the Black-Scholes model is not sufficient for a robust application. Consequently a lot of research has been made on new models that extend Black-Scholes considering at least one of the previous characteristics. These extended models often result in non-linear pricing equations, which we introduce below.

The Leland  model \cite{leland1985option} was one of the first and most popular extensions of the Black-Scholes to accommodate transaction costs. This model assumes only discrete trading at a pre-specified time intervals as opposed to the Black-Scholes model where trading is made continuously. 
Following Leland's approach, a model involving variable transaction costs has been introduced and analyzed in \cite{sevcoviczitnanska}. For an overview of nonlinear option pricing models of the Leland type under transaction costs we refer to \cite{DS2017}.

A model that considers proportional transaction costs was introduced by Barles and Soner in \cite{barles}. The authors apply a utility maximization approach, where they consider an economic agent with a constant risk-aversion. Using asymptotic analysis where the transaction costs were taken to zero and the risk-aversion to infinity they found that, in the limit, price of an European style option is given by a PDE of the Black-Scholes type where the volatility nonlinearly depends on the Gamma of the option price. 

The concept of a utility function and even expected utility has been known and used for several decades in economics in general and its usage to price financial derivatives has gained a lot of momentum in recent years. The idea was first formulated by Hodges and Neuberger \cite{hodges1989optimal}, however, their work not fully formalized and mathematically proved which was then done by other authors such as Davis and Zariphopoulou \cite{davis} that proposed a numerical scheme for solving the equation. 
Barles and Soner worked on that same model and provided an analytical study by take asymptotic limits on the transaction costs and risk-aversion coefficient (c.f. \cite{barles}). This has become a very well-known model due to many practical reasons. Indifference pricing theory was presented in the book \cite{utilitybook} by Carmona and {\c C}inlar which covers, in a very deep and comprehensive matter, the subject of pricing via utility maximization.

In \cite{davis} the authors investigated the problem of pricing European options in a Black-Scholes model with proportional costs on stock transactions and they defined the option writing price as the difference between the utilities achievable by going into the market to hedge the option and by going into the market on one’s own account. Without transaction costs, this definition is shown to yield the usual Black-Scholes price. To compute the option price under transaction costs, one has to solve two stochastic control problems, corresponding to the two utilities compared above. The value functions of these problems are shown to be the unique viscosity solutions of one fully nonlinear quasi-variational inequality, with two different boundary and terminal conditions. 
They constructed a stable and  convergent discretization scheme based on the binomial approximation of the stock price process. 
A generalization of this model was done by Cantarutti \emph{et al.} \cite{Cantarutti2017}, where besides having proportional transaction costs,  the underlying stock price dynamics was considered to have the form of a general exponential L\'evy process. Numerical results are obtained by Markov chain approximation methods when the returns follow a Brownian motion and a variance gamma process. 

In \cite{Monoyios2004}, an efficient algorithm is developed to price European options in the presence of proportional transaction costs, using the optimal portfolio framework of Davis et al. in \cite{MDSbook}.  In this approach, the fair option price is determined by requiring that an infinitesimal diversion of funds into the purchase or sale of options has a neutral effect on achievable utility. 
This results in a general option pricing formula, in which option prices are computed from the solution of the investor's basic portfolio selection problem, without the need to solve a more complex optimisation problem involving the insertion of the option payoff into the terminal value function.
The option prices are computed numerically using a Markov chain approximation to the continuous time singular stochastic optimal control problem, for the case of exponential utility. Comparisons with approximately replicating strategies are made. The method results in a uniquely specified option price for every initial holding of stock, and the price lies within bounds which are tight even as transaction costs become large. A general definition of an option hedging strategy for a utility maximising investor is developed that involved calculating the perturbation to the optimal portfolio strategy when an option trade is executed.

In \cite{Kallsen2015}, asymptotic formulas for utility indifference prices and hedging strategies in the presence of small transaction costs were obtained. In \cite{Perrakis2000} Perrakis and Lefoll  derived optimal perfect hedging portfolios in the presence of transaction costs. In the paper \cite{Yan2021} the price of a European option with proportional transaction costs has been determined using a utility indifference approach where the resulting Hamilton-Jacobi-Bellman equation for the portfolio without option is two-dimensional instead of three-dimensional as in standard utility indifference approaches (c.f. \cite{davis}).

Furthermore, Li and Wang \cite{li2009penalty} study the application of the penalty method to solve the resulting variational inequality. Song Wang and Wen Li have published numerical results of an implementation of the penalty method to price for both American and European style options (c.f. \cite{Li:2014:NMP:2660593.2660608}). They considered exponential utility which is by far one of the most studied utilities but still slightly restrictive.

Our goal is to analyze the system of two Hamilton-Jacobi-Bellman (HJB) equations.The option price is constructed as a difference of the certainty equivalents to the value functions solving the system of HJB equations. We introduce a transformation method for solving the penalized nonlinear partial differential equation. The transformed equation involves possibly non-constant and non-zero risk aversion function containing the negative ratio between the second and first derivatives of the utility function. Using the parabolic comparison principles we derive useful bounds on the option price. We also propose a finite difference numerical discretization scheme with some computational examples. 

The paper is organized as follows. The next section is focused on generalization of the utility indifference option pricing model. We consider a general class of concave utility functions. A system of two Hamilton-Jacobi-Bellman equations is derived. The option price is then obtained in terms of  a difference of their solutions. In Section \ref{sec3} we present a transformation method for solving the penalized nonlinear partial differential equation. The penalized equation involves the risk aversion function. Section \ref{sec4} is devoted to construction of a numerical scheme which is based on time implicit backward Euler method in combination with an upwind finite difference method for spatial discretizations. The Hamilton-Jacobi-Bellman equations are solved by means of the penalty method utilizing the policy iteration method. It contains numerical examples of option prices for various concave utility functions. 

\section{Utility Indifference Option Pricing Model}
 
In economics, a utility function is a function measuring the economic agent's preferences on different goods. In a financial context the utility function is usually applied to monetary quantities, and it can be used to measure the agent's risk aversion when in a context of uncertainty.

The usual requirements for an utility function are that it is continuous and non-decreasing function. Additional properties such as concavity or convexity can be shown to be directly related with the investor's risk aversion (see below).

\subsection{Risk Aversion and the Concept of a Certainty Equivalent}
If an investor's wealth at a future time $T$ is affected by a source of uncertainty then it can be modeled by a random variable, say $W_T$ which we assume to have finite expectation $\mathbb{E}[W_T]=: w$. Then, we know from Jensen's inequality that if $U:\mathbb{R}\to \mathbb{R}$ is a concave function then  
\begin{equation*}
\mathbb{E}[U(W_T)]  -U(w)\leq  0.
\end{equation*}
This 
 difference can be seen as how much an investor prefers to hold an uncertain amount (which can turn out to be greater or lower that its average) or its \emph{average}. The greater the concavity of $U$ the greater that difference, and that leads one to define the Arrow-Pratt measure of absolute risk-aversion (also referred to as  the coefficient of risk-aversion),
\begin{equation*}
R(\xi) \equiv - U''(\xi)/U'(\xi).
\end{equation*}
For a concave increasing utility function $U$ we have $R\ge0$. Now, for pricing financial options one still needs one more concept, the concept of certainty equivalent which we denoted by $v$. It is defined as follows:
\begin{equation*}
v: U(v)=\mathbb{E}[U(W_T)], \quad v= U^{-1}(\mathbb{E}[U(W_T)]).
\end{equation*}
Throughout the paper we shall consider various types of utility functions with different risk-aversion profiles. Their profiles $U(x)$ and the risk aversion coefficients\linebreak $R(\xi)=-U''(\xi)/U'(\xi)$ are shown in Table~\ref{utilities}.

\begin{table}
\small 
\caption{Utility functions, their inverse functions and risk aversion functions.
\label{utilities} 
}
\setlength{\tabcolsep}{2.88mm}

\begin{tabular}{lllll}
\toprule
\textbf{Type} & \textbf{Utility Function} & \textbf{Parameter} &  \textbf{Inverse Utility Function} & \textbf{Risk Aversion} \\
\midrule
Linear      & $U(\xi)=\xi$                  &   ---        & $U^{-1}(y) = y$ & $R(\xi)=0$ \\
Exp.  & $U(\xi)= 1- e^{-\gamma \xi}$  &  $\gamma>0$  & $U^{-1}(y) = -\ln(1-y)/\gamma$ & $R(\xi)=\gamma$ \\
Power       & $U(\xi)= \xi^a$  &  $a<1$     & $U^{-1}(y) = y^{1/a}$ & $R(\xi)=(1-a) \xi^{-1}$ \\
Log. & $U(\xi)= \ln (b \xi +1)$      &  $b>0$     & $U^{-1}(y) = (e^y-1)/b$ & $R(\xi)=b/(b \xi+1)$ \\
\bottomrule
\end{tabular}
\end{table}

The power utility function $U(\xi)= \xi^a$ and the logarithmic utility function\linebreak  $U(\xi)= \ln (b \xi +1)$ belong to the wide class of the so-called decreasing absolute risk aversion (DARA) utility function characterized by a decreasing risk aversion coefficient $R(\xi)$ considered as a function of the wealth $\xi$. In the context of dynamic stochastic portfolio optimization the importance of DARA utility functions has been investigated in papers by Kopa \emph{et al.}\cite{Kopa} and Kilianov\'a and \v{S}ev\v{c}ovi\v{c} \cite{KilianovaKybernetika}.

\subsection{Utility Indifference Option Pricing Model under Transaction Costs}
To price a derivative in a market with proportional transaction costs using this framework we proceed as follows. 

Firstly, we consider that we have an investor that can invest in either a risky asset or a risk-free asset following the dynamics of a geometric Brownian motion and exponential deterministic process, respectively. That is
\begin{equation}
dS = \mu S \text{dt} + \sigma S\text{dw}_t,
\qquad 
dB = r B dt, 
\label{riskfree}
\end{equation}
where $\{w_t, t\ge0\},$ stands for the standard Wiener stochastic process. 

We model the transactions costs by introducing two different prices for the risky asset, depending on weather the investor is buying (ask price) or selling (bid price) the underlying asset,
\[  
S_{ask} = (1+\theta) S, \ S_{bid} = (1-\theta) S, 
\]
where $\theta = (S_{ask} - S_{bid})/(2 S)$ represents the bid-ask spread factor. 

We assume that the investor can  buy or sell shares of the risky asset (shares account $\alpha_t$) by increasing or decreasing his holdings in the risk-free  asset (money account $\beta_t$). We model the cumulative purchase of risky assets by a process $L_t$ and the sale by the process $M_t$. The goal of an investor is to maximize his terminal utility and will chose his trades $L_t$, $M_t$ accordingly. The resulting dynamics for the investors portfolio are, 
\begin{equation}
d\alpha = dL-dM, \qquad d\beta = r\beta dt - (1+\theta) S dL + (1-\theta) S dM.
\label{ms_account}
\end{equation}
Now, we can define the liquid wealth $W_t$ of the investor as follows:
\begin{equation*}
	W_t = \beta_t + S_t(\alpha_t - \theta |\alpha_t|).
\end{equation*}
Then the option pricing problem can be formulated as a stochastic control problem. Firstly, we introduce a portfolio in which the investor is optimizing their expected utility by trading either stock or bonds, 
\begin{equation}
  v^0 (\alpha,\beta,s,t)= \sup_{L,M} \mathbb{E} \left[ U(W_T) | \alpha,\beta,s\right].
\label{vf}
\end{equation}
Secondly, we consider that the investor has at his disposal another portfolio, which is equally comprised of risk-free and risky assets but also a short buyer position with $\delta=-1$ or long seller position with $\delta=+1$ on a derivative with the terminal payoff $C_T$. The value $v^\delta$ of this second portfolio is given by,
\begin{equation}
  v^\delta (\alpha,\beta,s,t)= \sup_{L,M} \mathbb{E} \left[ U(W_T + \delta C_T(s)) | \alpha,\beta,s\right].
\label{v1}
\end{equation}
Let us denote the certainty equivalents of portfolios by $z^\delta=z^\delta(\alpha,\beta,s,t)$ and\linebreak $z^0=z^0(\alpha,\beta,s,t)$, respectively. The functions $z^\delta, z^0$ satisfy the system of equations:
\begin{equation}
U(w-z^0)=v^0 , \qquad U(w-z^\delta)=v^\delta, \quad w = \beta  + s(\alpha - \theta |\alpha|).
\label{zf1}
\end{equation}
For further details of utility indifference option pricing we refer to the book \cite{utilitybook}, the price $V$ of an option with a payoff diagram $C_T$ is given as the discounted difference of certainty equivalents, i.e., 
\begin{equation}
V=e^{-r(T-t)} (z^\delta-z^0), \quad \text{where}\ \ z^\delta-z^0 = U^{-1}(v^0)- U^{-1}(v^\delta).
\label{priceV}
\end{equation}
In order to determine solutions $z^\delta$ and $z^0$ we have to solve a pair of stochastic optimal control problems for $v^\delta$, and $v^0$, by means of the dynamic programming principle. Mathematical representation leads to a system of two Hamilton-Jacobi-Bellman (HJB) equations that we introduce and analyze in the next subsection.

\subsection{Hamilton-Jacobi-Bellman Equations for Value Functions}

Following \cite{davis} the functions $v^\delta$ and $v^0$ satisfy the system of variational inequalities of the form:
\begin{equation}
\min\left(\mathcal{V}_A (v),\mathcal{V}_B (v), \mathcal{V}_C (v)\right)=0,
\label{minimalequation}
\end{equation}
where the linear differential operators $\mathcal{V}_A,\mathcal{V}_B,\mathcal{V}_C$ are defined as follows:
\[
\mathcal{V}_A \equiv \partial_t + \frac{\sigma^2}{2}s^2 \partial^2_{s}+\mu s \partial_s + r\beta  \partial_\beta,
\ \  
\mathcal{V}_B \equiv -\partial_\alpha+s(1+\theta)\partial_\beta,
\ \  
\mathcal{V}_C \equiv \partial_\alpha-s(1-\theta)\partial_\beta, 
\]
and  terminal conditions,
\begin{equation}
 v^0(\alpha,\beta,s,T)=U(w(\alpha,\beta,s)), \qquad  v^\delta(\alpha,\beta,s,T)=U(w(\alpha,\beta,s) +\delta C_T(s)).
\end{equation}
Here $C_T(S)$ denotes the prescribed pay-off diagram, i.e., $C_T(S)=(S-K)^+$ in the case of a plain vanilla call option, or $C_T(S)=(K-S)^+$ in the case of a put option. Here $K>0$ denotes the strike price, $(\beta)^+= \max(\beta,0)$, and $(\beta)^-=\min(\beta,0)$ denote the positive and negative parts of a real number $\beta$. Recall that the no-transaction region of values $(\alpha,\beta,s,t)$ is characterized by the equation $\mathcal{V}_A (v)=0$. The Buy and Sell regions correspond to equations $\mathcal{V}_B (v)=0$ and $\mathcal{V}_C (v)=0$, respectively (c.f. \cite{davis}). 

The minimal Equation (\ref{minimalequation}) is equivalent to the following linear complementarity problem for the functions $v=v^0$, and $v=v^\delta$:
\begin{equation}
\mathcal{V}_A (v) \geq 0,\ \mathcal{V}_B (v) \geq 0,\ \mathcal{V}_C (v) \geq 0,
\qquad 
\mathcal{V}_A (v)\cdot \mathcal{V}_B (v)\cdot \mathcal{V}_C (v)=0.
\label{lincomplementarity}
\end{equation}

\subsection{Penalty Method for Solving HJB Equations}

With penalty methods, the initial variational inequality is replaced by one single equation which has a term parameterized by a small parameter. One should prove that the solution of this new equation will converge to the initial one. Besides convergence to the initial problem, the perturbed equation’s solution will always respect the constrains posed by the initial problem. Next, introduce the specific implementation of the penalty method for our pricing model following \cite{li2009penalty}; \cite{Li:2014:NMP:2660593.2660608}. The penalty method has been successfully adopted for solving various nonlinear option pricing model by Lesmana and Wang \cite{lesmana}, or Chernogorova and Valkov \cite{chernogorova}, and others. The optimal time dependent penalty function has been proposed recently by Clevenhaus \emph{at al.} \cite{clevenhaus}.

Next we introduce the penalty method in a more detail. Let us define the following penalized perturbed equation for the function $v=v_{\lambda_B,\lambda_C}(\alpha,\beta,s,t)$, 
\begin{equation}
\mathcal{V}_A (v) + \lambda_B [\mathcal{V}_B (v)]^- +  \lambda_C [\mathcal{V}_C (v)]^-=0.
\label{penalized}
\end{equation}
Here  $\lambda_B,\lambda_C \gg 0$ are sufficiently large penalty parameters. 
In what follows, we will drop the subscripts for the sake of simplicity $v:=v_{\lambda_B,\lambda_C}$. In the limit $\lambda_B, \lambda_C\to \infty$ we formally deduce that the limiting solution $v$ solves the linear complementarity problem. Indeed, $\mathcal{V}_A (v) = - \lambda_B [ \mathcal{V}_B (v)]^- - \lambda_C[ \mathcal{V}_C (v)]^- \ge 0$, and $\mathcal{V}_B (v), \mathcal{V}_C (v)\ge 0$ in the limit $\lambda_B, \lambda_C\to \infty$. Taking $\lambda_B, \lambda_C\to \infty$ such that $\lambda_B/\lambda_C\to \infty$ we obtain $\mathcal{V}_A (v) \mathcal{V}_B (v) =0$. Similarly, $\mathcal{V}_A (v) \mathcal{V}_C (v) =0$.

\section{Transformation of the HJB Equation Involving Risk Aversion Function}\label{sec3}

For a general utility function $U$ we search the solution $v=v^\delta, \delta=0,\pm1$, in the following form:
\[
v(\alpha,\beta,s,t) = U(e^{\mu (T-t)} (w(\alpha,\beta,s) + \mathscr{A}(t) + \mathscr{V}(\alpha,\beta, s,t)) ),
\]
where
\[
\mathscr{A}(t) = \frac{\beta}{\mu}(r-\mu) ( 1- e^{-\mu(T-t)} )
\]
is a time dependent shift function such that  $\mathscr{A}(T)=0$. Notice that $\partial^2_s w=0, \partial_\beta w=1$, and $s \partial_s w = w-\beta$. Hence 
\begin{eqnarray*}
\mathcal{V}_A (e^{\mu (T-t)} (w + \mathscr{A} ) )
&=& e^{\mu (T-t)} ( -\mu(w+ \mathscr{A}) + \mathscr{A}^\prime + \mu(w-\beta) +r\beta)
\\
&=& e^{\mu (T-t)} ( \mathscr{A}^\prime  -\mu\mathscr{A}  + \beta(r-\mu)) = 0,
\end{eqnarray*}
because of the definition of the auxiliary function $ \mathscr{A}(t)$. For any function $z=z(\alpha,\beta,s,t)$ we have 
\begin{eqnarray*}
\mathcal{V}_A (U(z))&=& U^\prime(z) ( \partial_t z +\mu s \partial_s z + r\beta\partial_\beta z
+ \frac{\sigma^2}{2} s^2 \partial_s(U^\prime(z) \partial_s z))
\\
&=& U^\prime(z) \left( \mathcal{V}_A (z) - \frac{\sigma^2}{2} R(z) (s\partial_s z))^2\right), 
\end{eqnarray*}
where $R(z)=-U''(z)/U(z)\ge 0$ is the risk aversion function. Taking $z=e^{\mu (T-t)} (w + \mathscr{A} + \mathscr{V})$ we obtain 
\[
\mathcal{V}_A (v) =
\mathscr{F} (\alpha,\beta,s, t) \biggl( \mathcal{V}_A (\mathscr{V}) -\mu \mathscr{V} -  \frac{\sigma^2}{2} R(z) e^{\mu (T-t)}  (s\partial_s w +  s \partial_s\mathscr{V})^2 
\biggr), 
\]
where $\mathscr{F} (\alpha,\beta,s, t) = U'(z) e^{\mu (T-t)} >0$ is a positive factor. For $\mathcal{V}_B (w)$, and $\mathcal{V}_C (w)$ we have
\[
\mathcal{V}_B (w) = \left\{
\begin{array}{rl}
     2s \theta, & \ \ \text{if} \ \alpha>0,  \\
     0        , & \ \ \text{if} \ \alpha\le0,  
\end{array}
\right. \quad 
\mathcal{V}_C (w) = \left\{
\begin{array}{rl}
     0          , & \ \ \text{if} \ \alpha>0,  \\
     2s \theta  , & \ \ \text{if} \ \alpha\le0 .
\end{array}
\right.
\]
Furthermore, as $ \mathcal{V}_B (\mathscr{A})=\mathcal{V}_C (\mathscr{A})=0$, we have
\[
\mathcal{V}_B (v) = \mathscr{F} (\alpha,\beta,s, t) ( \mathcal{V}_B (w) +  \mathcal{V}_B (\mathscr{V}) ),
\quad 
\mathcal{V}_C (v) = \mathscr{F} (\alpha,\beta,s, t) ( \mathcal{V}_C (w) + \mathcal{V}_C (\mathscr{V}) ).
\]
A solution $\mathscr{V}=\mathscr{V}^\delta$ is subject to the terminal condition at $t=T$:
\begin{equation}
\mathscr{V}(\alpha,\beta,s,T) = \left\{
\begin{array}{rl}
     0,         & \ \ \text{if} \ \delta=0,  \\
     C_T(s), & \ \ \text{if} \ \delta=1, \ \ \text{(long seller position)}, \\
     -C_T(s), & \ \ \text{if} \ \delta=-1,\ \ \text{(short buyer position)}.  \\
\end{array}
\right.
\label{terminal}
\end{equation}
Summarizing, we deduce that the penalized problem (\ref{penalized}) can be reformulated in terms of the  function $\mathscr{V}$ as follows:
\begin{eqnarray}
\mathcal{V}_A (\mathscr{V}) - \mu\mathscr{V} 
&-&  \frac{\sigma^2}{2} R(z) e^{\mu (T-t)}  (s\partial_s w +  s \partial_s\mathscr{V})^2
\nonumber 
\\
&+& \lambda_B [\mathcal{V}_B (w + \mathscr{V})]^- +  \lambda_C [\mathcal{V}_C (w +\mathscr{V})]^-=0.
\label{penalized-V}
\end{eqnarray}

Recall that the option price $V$ is obtained as the difference between certainity equivalents $z^0$ and $z^\delta$. It means that 
\[
V = e^{-r(T-t)} (z^\delta-z^0) = e^{-r(T-t)}( U^{-1}(v^0)- U^{-1}(v^\delta)) 
= e^{(\mu-r)(T-t)}( \mathscr{V}^0- \mathscr{V}^\delta).
\]
In the next proposition we compare a solution to the system of transformed HJB equations with the explicit solution to the linear Black-Scholes equation:
\[
\partial_t \mathscr{V}+ \frac{\sigma^2}{2}s^2 \partial^2_{s} \mathscr{V}+ \mu s \partial_s \mathscr{V} - \mu\mathscr{V} = 0, \qquad \mathscr{V}(s,T) = \delta C_T(s), \quad \delta=0, \pm1.
\]
In the call option case where $C_T(s)=(s-K)^+$ the price $\mathscr{V}(s,t)=\mathscr{V}_{BS}(s,t)$ is  given by an explicit formula:
\[
\mathscr{V}_{BS}(s,t) = \delta\left( s \Phi(d_1) -  K e^{-\mu(T-t)} \Phi(d_2) \right), 
\]
where $d_{1,2} = (\ln(s/K) + (\mu\pm\sigma^2/2) (T-t))/(\sigma\sqrt{T-t})$ and $\Phi(d)=(2\pi)^{-1/2} \int_{-\infty}^d e^{-\xi^2/2} d\xi$. A similar formula is available for pricing of put options. 

\begin{proposition}\label{prop1}
Assume $U$ is an exponential ($\gamma>0$) or linear ($\gamma=0$) utility function, i.e., its risk aversion function $R(\xi)=-U''(\xi)/U'(\xi) \equiv \gamma$ where $\gamma\ge 0$ is a non-negative constant. 
Then the solution $\mathscr{V}$ of the penalized problem (\ref{penalized-V}) satisfying the terminal condition (\ref{terminal})  is independent of the factor $\beta$, i.e., $\mathscr{V}=\mathscr{V}(\alpha,s,t)$. Consequently, the option price $V=V(\alpha,s,t)$ is independent of $\beta$.~Moreover,
\[
\mathscr{V}(\alpha,s,t) \le \mathscr{V}_{BS}(s,t), \quad \text{for all}\ \alpha, s>0, t\in[0,T].
\]
\end{proposition}

\begin{proof}
Notice that  $R(\xi)\equiv\gamma$, and $s\partial_s w, \mathcal{V}_B(w), \mathcal{V}_C(w)$, as well as the terminal condition~\mbox{(\ref{terminal})} are independent functions of the factor $\beta$. The penalized Equation (\ref{penalized-V}) can be rewritten \mbox{as follows:}
\begin{eqnarray}
&&\partial_t \mathscr{V} + \frac{\sigma^2}{2}s^2 \partial^2_{s} \mathscr{V}+ \mu s \partial_s \mathscr{V} + \mu \beta \partial_\beta \mathscr{V} - \mu\mathscr{V}
\label{penalized-V2}
\\
&& =  \frac{\sigma^2}{2} \gamma e^{\mu (T-t)}  (s\partial_s w +  s \partial_s\mathscr{V})^2
- \lambda_B [\mathcal{V}_B (w) + \mathcal{V}_B(\mathscr{V})]^- -  \lambda_C [\mathcal{V}_C (w) +\mathcal{V}_C(\mathscr{V})]^-.
\nonumber 
\end{eqnarray}
The right-hand side of (\ref{penalized-V2}) is nonnegative and it does not explicitly depend on $\beta$, so does the solution $\mathscr{V}=\mathscr{V}(\alpha,s,t)$. Furthermore,
\[
\partial_t \mathscr{V} + \frac{\sigma^2}{2}s^2 \partial^2_{s} \mathscr{V}+ \mu s \partial_s \mathscr{V}  - \mu\mathscr{V} \ge 0.
\]
As the Black-Scholes solution $\mathscr{V}_{BS}$ satisfies 
\[
\partial_t \mathscr{V}_{BS} + \frac{\sigma^2}{2}s^2 \partial^2_{s} \mathscr{V}_{BS}+ \mu s \partial_s \mathscr{V}_{BS} - \mu\mathscr{V}_{BS} = 0, 
\]
then, taking into account $\mathscr{V}(\alpha, s,T)=\mathscr{V}_{BS}(s,T)$, applying the maximum principle for parabolic equations on unbounded domains due to Meyer and Needham (\cite{Meyer}, \mbox{Theorem 3.4}), 
 we obtain the inequality  $\mathscr{V}(\alpha,s,t) \le \mathscr{V}_{BS}(s,t)$ for a given  parameter $\alpha$ and all \mbox{$s>0, t\in[0,T]$,} as claimed.
\end{proof}

The following proposition is a direct consequence of Proposition~\ref{prop1}.

\begin{proposition}\label{prop2}
Assume $U$ is the linear utility function, i.e., its risk aversion function $R(\xi)\equiv 0$. Suppose that there are no transaction costs, i.e., $\theta=0$.
Then the solution $\mathscr{V}$ of the penalized problem~(\ref{penalized-V}) satisfying the terminal condition~(\ref{terminal}) is independent of the factors $\alpha, \beta$, i.e.,\linebreak $\mathscr{V}=\mathscr{V}(s,t) = \mathscr{V}_{BS}(s,t)$, i.e., $\mathscr{V}$ is the Black-Scholes price of a European style option.
\end{proposition}

\begin{proof}
Since $\theta=0$ we have $\mathcal{V}_B(v)=-\mathcal{V}_C(v)$ for any function $v$.  Hence $v$ is a solution to (\ref{lincomplementarity}) if and only if $\mathcal{V}_B(v)=\mathcal{V}_C(v)=0$ and $\mathcal{V}_A(v)\ge 0$. Therefore, a solution $\mathscr{V}$ to the penalized problem (\ref{penalized-V}) satisfies the linear Black-Scholes equation. Hence, $\mathscr{V}=\mathscr{V}(s,t) = \mathscr{V}_{BS}(s,t)$, \mbox{as claimed.}
\end{proof}

\section{Construction of a Numerical Discretization Upwind Finite Difference Scheme}\label{sec4}

In this section we propose a numerical discretisation scheme and several computational examples.  The scheme is based on the finite difference method proposed in \mbox{\cite{li2009penalty}}. The resulting scheme is of upwind type in the space discretization and the backward Euler implicit scheme in time. 

\subsection{Finite Difference Approximation of a Solution to the Penalized Problem}

We first introduce $\Omega^b$ the  truncated domain corresponding to the solvency region where $\beta+S(\alpha-\theta|\alpha|) >0$ as follows:
\[
\Omega^b =\{ (\alpha,\beta,S) 
\in (L^-_\alpha, L^+_\alpha) \times (L^-_\beta, L^+_\beta) \times (0, S^+): \beta+S(\alpha-\theta|\alpha|) >0 \}.
\]
We consider a simple 3D uniform mesh grid:
\[
(\alpha_i, \beta_j, S_k)\in\Omega^b, \quad i=0, \cdots, N_\alpha,\quad 
j=0, \cdots, N_\beta, \quad k=0, \cdots, N_S,
\]
\[
\alpha_i = L^-_\alpha + i h_\alpha, \quad  \beta_j = L^-_\beta + j h_\beta, \quad  S_k =  k h_S,
\]
\[
h_\alpha = (L^+_\alpha - L^-_\alpha)/N_\alpha,
\quad
h_\beta = (L^+_\beta - L^-_\beta) /N_\beta,
\quad
h_S = S^+/N_S,
\]
where $N_\alpha, N_\beta$, and $N_S$ are the numbers of discretization steps in the $\alpha,\beta$, and $S$ variables. Notice that the spatial discretization can be easily adopted to a non-uniform grid, e.g., by considering uniform discretization for the logarithmic variable $x_k=\ln(S_k/K)$. We consider a uniform time discretization with time steps $n \Delta t$ for $n=N, \cdots, 1,0$, where $\Delta t = T/N$, and $N$ is the number of time discretization steps. The solution $v = v(\alpha,\beta,S,t)$ will be discretized by the value $V^n_{ijk}$ at $(\alpha_i, \beta_j, S_k)$ and time $t= n \Delta t$.

We define the following finite difference discretization operators:
\begin{equation}
\begin{aligned}
\label{discretization operators}
&& D_t V_{ijk}^n = \frac{V_{ijk}^{n+1}-V_{ijk}^{n}}{\Delta t},
\quad D_\alpha^\pm V_{ijk}^n = \pm \frac{V_{(i\pm 1)jk}^{n}-V_{ijk}^{n}}{h_\alpha}, \quad 
D_\beta^\pm V_{ijk}^n = \pm \frac{V_{i(j\pm1)k}^{n}-V_{ijk}^{n}}{h_\beta}.
\\
&&
D_S  V_{ijk}^n = \frac{V_{ij(k+1)}^{n}-V_{ijk}^{n}}{h_S}, 
\  D_{SS}  V_{ijk}^n = \frac{V_{ij(k+1)}^{n}-2V_{ijk}^{n}+V_{ij(k-1)}^{n}}{h_S^2},
\end{aligned}
\end{equation}
\begin{equation}
\label{discreteL}
\begin{split}
&\mathcal{L_A}V_{ijk}^n = -\left( D_t +r (\beta_j)^+ D_\beta^+ + r (\beta_j)^- D_\beta^-  + \mu S_k D_S + \frac{\sigma^2}{2}S_k^2 D_{SS}\right) V_{ijk}^n,  
\\
&\mathcal{L_B}V_{ijk}^n = \left( -D_\alpha^+ + (1+\theta)S_k D^-_\beta\right) V_{ijk}^n, \quad 
\mathcal{L_C}V_{ijk}^n =  \left( D_\alpha^- -(1-\theta)S_k D^+_\beta\right) V_{ijk}^n. 
\end{split}
\end{equation}
Clearly, for any $\lambda>0$, and $\mathcal{L}\in\mathbb{R}$, we have 
\[
\lambda [\mathcal{L}]^- 
= \min_{m\in[0,\lambda]} m\mathcal{L}
= \left\{
\begin{array}{rl}
     0,                  & \ \ \text{if} \ \mathcal{L}>0,  \\
     \lambda\mathcal{L}, & \ \ \text{if} \ \mathcal{L}\le0.  
\end{array}
\right.
\]
The numerical discretization scheme is then given by:
\begin{equation}
\mathcal{L_A} V_{ijk}^n + \min_{\bar m \in [0,\lambda_B]}\bar m \mathcal{L_B} V_{ijk}^n + \min_{\bar n \in [0,\lambda_C]}\bar n \mathcal{L_C} V_{ijk}^n = 0, \quad \forall i,j,k.
\label{system}
\end{equation}

\noindent Terminal conditions. 

\noindent For the last terminal time level $n=N$ we have, for the call (put) option case with the pay-off diagram $C_T(S)=(S-K)^+ \ ( C_T(S)=(K-S)^+)$,
\[
V_{ijk}^N =  U(\beta_j + S_k (\alpha_i-\theta |\alpha_i|) + \delta C_T(S_k) ), \ \text{for} \ for (\alpha_i, \beta_j, S_k)\in\Omega^b.
\]

\noindent Boundary conditions.
\\noindent We apply the Dirichlet boundary conditions, i.e.,
\[
V_{ijk}^n = U(\beta_j + S_k (\alpha_i-\theta |\alpha_i|) + \delta C_T(S_k) ), 
\ \text{for}\  (\alpha_i, \beta_j, S_k)\in\partial \Omega^b, \ n=N-1,\cdots,1,0.
\]
Here we set $\delta=0$ in the case of numerical approximation of the value function $v^0$, and $\delta=\pm 1$ in the case of approximation of the value function $v^\delta, \delta=\pm 1$. 

Next, we present the full numerical discretization algorithm involving the policy iteration method for solving the penalized PDE (\ref{system}).
Notice that $(\beta_j)^+ + (\beta_j)^- = \beta_j$. It yields the following system of linear equations for the unknown vector $V^n$ for \mbox{$n=N-1, \cdots, 1, 0$}, 
\begin{eqnarray}
\label{linearsystem}
&& 	
\biggl[1 +\Delta t\left( 
\frac{r}{h_\beta} \beta_j +\frac{\mu}{h_S} S_k +\frac{\sigma^2}{h_S^2}S^2_k 
\right)
\biggr] 
V^{n,p+1}_{ijk}
\nonumber
\\
&&
-\biggr[ \frac{\mu \Delta t}{h_S} S_k +\frac{\sigma^2 \Delta t}{2h_S^2} S_k^2\biggl] V^{n,p+1}_{ij(k+1)}
-\frac{\sigma^2 \Delta t}{2h_S^2} S^2_k  V^{n,p+1}_{ij(k-1)}
\nonumber
\\
&& 	
+\Delta t\biggl[
\tilde{m}\left(\frac{1}{h_\alpha}+\frac{(1+\theta)}{h_\beta}S_k\right)
+ \tilde{n}\left(\frac{1}{h_\alpha}+\frac{(1-\theta)}{h_\beta}S_k\right)\biggr]
V^{n,p}_{ijk}
\\
&&
-\Delta t \biggl[ \frac{r\beta_j^+}{h_\beta}+\tilde{n}\frac{(1-\theta)}{h_\beta} S_k \biggl] V^{n,p}_{i(j+1)k}
- \Delta t\biggl[ \frac{r\beta_j^-}{h_\beta}+\tilde{m}\frac{(1+\theta)}{h_\beta}S_k \biggr] V^{n,p}_{i(j-1)k}
\nonumber 
\\
&& -\tilde{m}\frac{\Delta t}{h_\alpha} V^{n,p}_{(i+1)jk} 
-\tilde{n}\frac{\Delta t}{h_\alpha} V^{n,p}_{(i-1)jk}
 = V^{n+1}_{ijk},
 \nonumber
\end{eqnarray}
where $p=0,\cdots, p_{max}$ is the policy iteration parameter, $\tilde m=\tilde m^{n,p}_{ijk}, \tilde n=\tilde n^{n,p}_{ijk}$ are arguments of the minimum in (\ref{system}).  The above system of linear equations for the unknown stacked vector $V=(V^{n,p+1}_{ijk})$ can be rewritten as a system of linear equations of the form $\mathcal{A} V = b$ where $\mathcal{A}$ is a sparse matrix with at most $3$ nonzero elements in each row. The right-hand side vector $b$ consists of the known vector $V^{n+1}$ complemented by the boundary conditions.
It is important to note that the coefficients of the matrix $\mathcal{A}$  depend on the coefficients $\tilde{m},\tilde{n}$ which has to be computed within each policy iteration step. The full algorithm for the computation of the value function is as in the Algorithm~\ref{algorithm}. 

\begin{algorithm}
\caption{The algorithm for computing the value function $V$ for $\delta=0,\pm 1$.}\label{algorithm}
\small
\SetAlgoLined

Initialization of model parameters and numerical parameters $p_{max}, tol_{max}>0$\;

Compute the terminal conditions for $n=N$.
\\
Set $V^N_{i,j,k}=U(w_{i,j,k} + \delta C_T(S_k))$ for each $i,j,k$\;
Set $n=N-1$;

 \While{$n>0$}{
  Initiate policy iteration $p=0$ with $V^{n,0}_{ijk}= V^{n+1}_{ijk}$\;
 
 Compute the right-hand side vector $b$ from $V^{n+1}_{ijk}$ and boundary conditions \;
 
 \While{$p<p_{max}$ and $tol \ge tol_{max}$ }{

Compute the penalty terms. Set $\tilde m^{n,p}_{ijk}=0, \tilde n^{n,p}_{ijk}=0$ 

\If{$\mathcal{V}_B (V^{n,p}_{ijk})<0$}{$\tilde m^{n,p}_{ijk}=\lambda_B$} 

\If{$\mathcal{V}_C (V^{n,p}_{ijk})<0$}{$\tilde n^{n,p}_{ijk}=\lambda_C$}

Compute the elements of the matrix $\mathscr{A}$\;

Solve the linear system of equations $\mathscr{A} V = b$\;

Compute the difference $tol = \max|V^{n,p}_{ijk} -  V_{ijk}|$\; 

$V^{n,p+1}_{ijk} \gets V_{ijk}$\; $p \gets p+1$
}
 
 $n \gets n-1$\;
 }

\end{algorithm}

\begin{remark}\label{rem-1}
Our numerical scheme is based on solving the system of linear Equation (\ref{system}) for the unknown stacked vector $(V_{ijk}^n)$. Its matrix representation contains at most 3 nonzero elements in each row. It has a block matrix structure with $N_\alpha\times N_\beta$ tridiagonal $N_S\times N_S$ matrices on the block diagonal. For each $N_S\times N_S$ tridiagonal we can employ the fast Thomas algorithm with time complexity $O(N_S)$. The overall complexity of computation of the system is therefore $O(N_\alpha\times N_\beta\times N_S \times p_{max})$ where $p_{max}$ is the maximal number of policy iterations. 
Recall that there are other fast and robust numerical methods for solving problems of the form (\ref{system}). Among them  there are alternating direction explicit (ADI) methods for linear, nonlinear and multi-dimensional Black-Scholes models (c.f. \cite{BEGP2017} and references therein).
\end{remark}

\subsection{Results of Numerical Approximation of Option Prices}

In this part, we present results of computation of European style call options for the exponential utility function with a risk parameter $\gamma>0$ and linear utility function. 

The model and numerical parameters used can be found in Table \ref{parameters}. We used the exponential mesh $S_k= K \ln(x_k)$ where $\{x_k, k=1,\cdots,  N_S\}$ is an equidistant mesh of the interval $[K/2, 2K]$. The plot of the call option price as a function of the underlying asset price $S$ is shown in Figure~\ref{fig-linear-exponential}, for the buyer call option prices, i.e., $\delta=-1$. We used Matlab framework for computation of the solution on 3GHz Intel single Core machine. For numerical discretization parameters shown in Table~\ref{parameters} The computational time was 8.1~sec per one policy iteration. According to Remark~\ref{rem-1} the overall complexity is of the order $O(N_\alpha\times N_\beta\times N_S \times p_{max})$. 


\begin{table}
\small \caption{Model and numerical parameters of the numerical solution.
\label{parameters}}
\setlength{\tabcolsep}{4.15mm}

\begin{tabular}{llclcl}
\toprule
\textbf{Model Parameters} & \textbf{Value} &\textbf{Num Params} & \textbf{Value} &\textbf{Num Params} & \textbf{Value} \\
\midrule
 Strike price $K$           &  50   & $N_\alpha$   & 6     &  $N_S$ & 100 \\
 Transaction costs $\theta$ & 0.01  & $L_\alpha^-$ & 0.2   &  $S^+$ & 100  \\
 Volatility $\sigma$        & 0.3   & $L_\alpha^+$ & 0.6   & $\lambda_B=\lambda_C$   & 10  \\
 Risk-free rate $r$         & 0.05  & $N_\beta$     & 6    &  $N$   &  10 \\
 Drift $\mu$                & 0.1   & $L_\beta^-$   & $-$100 &  $T$   & 1  \\
  Risk-aversion $\gamma$    & 0.1   & $L_\beta^+$   & 100  &   &   \\
\bottomrule
\end{tabular}

\end{table}

\vspace{-6pt}

\begin{figure}
\centering
	\includegraphics[width=0.45\textwidth]{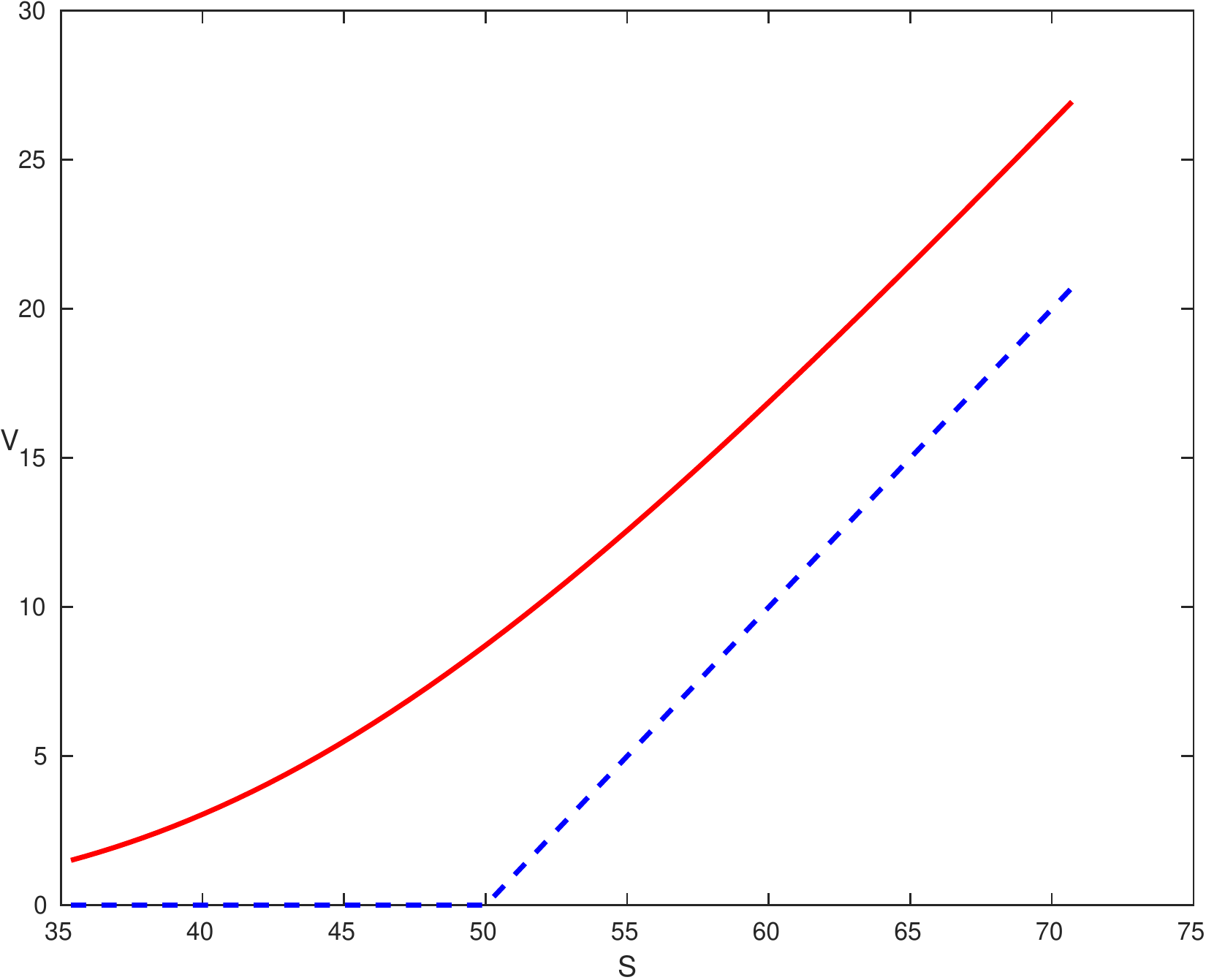}
 	\includegraphics[width=0.45\textwidth]{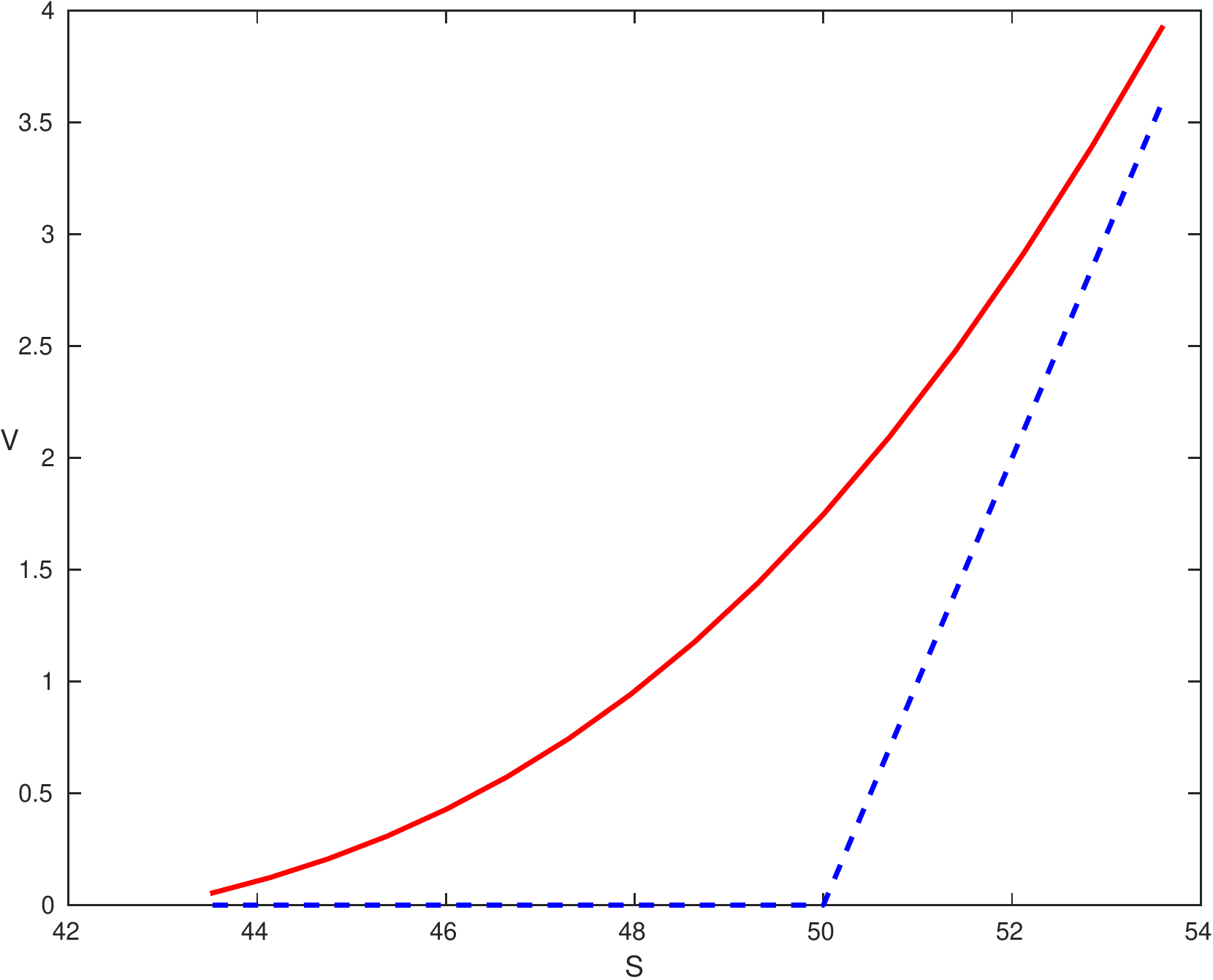}

	\caption{Linear (\textbf{left}) and exponential (\textbf{right}) utility indifference call option buyer price as a function of the underlying asset price $S$ with the wealth function parameters  $\alpha=0.467, \beta=33.3$. }
 	\label{fig-linear-exponential}
\end{figure}

\section{Conclusions}
In this paper we investigated and  analyzed the system of Hamilton-Jacobi-Bellman equations arising in pricing financial derivatives. We followed the utility indifference option pricing model in which the option price is constructed as a difference of the certainty equivalents to the value functions solving the system of HJB equations. We analyzed solutions to the transformed nonlinear partial differential equation involving a possibly non-constant risk aversion function. Useful bounds on the option price were obtained using parabolic comparison principle.  We also proposed a finite difference numerical discretization scheme. Various computational examples were also presented.

\vspace{6pt} 

\noindent{\bf Acknowledgments}
The authors gratefully acknowledge the contribution of the Slovak Research and Development Agency under the project APVV-20-0311.


\end{document}